\documentclass[12pt]{iopart}

\usepackage{epstopdf}
\usepackage{subfigure}

\usepackage{latexsym,bm}
\usepackage{xcolor,graphicx}
\usepackage{eso-pic}
\usepackage{caption}
\usepackage{subfigure}
\newtheorem{lemma}{Lemma}[section]
\newtheorem{theorem}{Theorem}[section]

\newtheorem{proof}{Proof}

\newtheorem{example}{Example}[section]

\makeatletter

\@addtoreset{equation}{section}
\makeatother

\begin{document}

\title{Numerical analysis on boundary integral equation to exterior Dirichlet problem of Laplace equation}
\author{Yidong Luo}
\address{School of Mathematics and Statistics, Wuhan University, Hubei Province, P. R. China}
\ead{Sylois@whu.edu.cn}
\vspace{10pt}
\begin{indented}
\item[]August 2017
\end{indented}

\begin{abstract}
This paper investigate on numerical analysis on modified Single-layer approach to exterior Dirichlet problem of Laplace equation. We complete the convergence and error analysis of Petrov-Galerkin and Galerkin-Collocation methods with trigonometric basis for the induced modified Symm's integral equation of the first kind on analytic boundary. Besides, utilizing the composite trapezial quadrature formula and trigonometric interpolation to handle the singularity in modified logarithmic kernel, we establish the numerical procedure for implementation. On these numerical examples, we compare the effect and efficiency of different Petrov-Galerkin and Galerkin-Collocation methods.
\end{abstract}
%
%
%
%
%

\section{Introduction }
Integral equation method plays an important role in solving the (BVP) of Laplace equations.
Let $ \Omega \subseteq \mathrm{R}^2 $ be bounded and simply connected with boundary $ \partial \Omega $ of class $ C^2 $ and $ f \in C( \partial \Omega) $. To solve Dirichlet problem of Laplace equation
\begin{equation*}
\Delta u = 0 \ \textrm{in} \ \Omega, \ u = f \ \textrm{on} \ \partial \Omega,
\end{equation*}
When $ f \in C^{1,\alpha} (\partial \Omega) $, the solution $ u $ can be represented as single-layer potential
\begin{displaymath}
 u(x) := - \frac{1}{2\pi}\int_{\partial \Omega} \psi (y) \ln \vert x-y \vert ds(y), \quad x \in \mathrm{R}^2,
\end{displaymath}
provided that the density $ \psi \in C^{0,\alpha}(\partial \Omega) $ solves
\begin{equation}
S \psi:= - \frac{1}{2\pi} \int_{\partial \Omega} \psi (y) \ln \vert x-y \vert ds(y) = f(x), \quad x \in \partial \Omega,
\end{equation}
\indent (1.1) is known as Symm's integral equation of the first kind. There exists numerous work on numerical solution of (SIE). Frequently used method is Petrov-Galerkin and collocation methods, for example, Galerkin and collocation boundary element method, see [1,2,3,4,9,10]; Cheybshev polynomial-based collocation method see [17,19]; spline Galerkin and collocation method, see [5,7,8,20,22]; piecewise constant collocation and Galerkin methods, see [8,12,18,21]; wavelet-based or trigonometric-based Galerkin method, see [13, Chapter 3.3] and [11].
\newline \indent For the simplicity and completeness of analysis, we are interested in the numerical analysis of projection methods under Fourier basis for planar (SIE), such as Petrov-Galerkin methods and Galerkin-Collocation method. In past, assuming $ \partial \Omega $ to be analytic with nonzero pointwise tangent, that is, $ \partial \Omega $ possesses the analytic parameterizations
  \begin{equation}
\partial \Omega:= \{ \gamma(t): t \in [0, 2\pi) \}
\end{equation}
and $ \vert \gamma'(t) \vert > 0, \ t \in [0, 2\pi) $. Inserting (1.2) into (1.1), (SIE) is transformed into integral equation of $ 1 $ D:
\begin{equation}
 - \frac{1}{2 \pi} \int^{2\pi}_0 \Psi (s) \ln \vert \boldsymbol{\gamma}(t) - \boldsymbol{\gamma}(s) \vert ds = g(t), \quad x \in [0,2\pi],
\end{equation}
for the transformed density $ \Psi(s) := \psi(\boldsymbol{\gamma}(s))\vert \dot {\boldsymbol{\gamma}}(s)\vert $ and $ g(t) := f(\boldsymbol{\gamma}(t)), \ s \in [0,2\pi] $. Complete convergence and error are obtained (See [13] and references therein) for Petrov-Galerkin methods and Galerkin-Collocation method to $ g \in H^r(0,2\pi), r \geq 1 $ under $ L^2 $ setting.

\indent Similar to interior Dirichlet problem, to solve exterior problem
\begin{equation*}
\Delta u = 0 \quad \textrm{in} \ \mathrm{R}^2 \setminus \overline{\Omega}, \quad u = f \quad \textrm{on} \ \partial \Omega,
\end{equation*}
\begin{equation*}
\quad u(x) = u_{\infty} + o(\frac{1}{\vert x \vert }), \quad \textrm{for}  \  \vert x \vert \to \infty.
\end{equation*}
When $ f \in C^{1,\alpha} (\partial \Omega) $, with introduction of mean value operator $ M $ defined by
\begin{equation*}
M: \varphi \mapsto \frac{1}{\vert \partial \Omega \vert} \int_{\partial \Omega} \varphi ds,
\end{equation*}
the solution $ u $ can be represented as the modified single-layer potential
\begin{displaymath}
 u(x) := - \frac{1}{2\pi} \int_{\partial \Omega} (\varphi (y) - M \varphi) \ln \vert x-y\vert ds(y) + M \varphi, \quad x \in \mathrm{R}^2,
\end{displaymath}
provided that the density $ \varphi \in C^{0,\alpha} (\partial \Omega) $ solves the integral equation
\begin{equation}
S_0 \varphi:= - \frac{1}{2\pi} \int_{\partial \Omega} (\varphi (y) - M \varphi) \ln \vert x-y\vert ds(y) + M \varphi = f(x), \quad x \in \partial \Omega,
\end{equation}
Notice that $ S_0 $ is injective on $ C(\partial \Omega) $ even with no specific geometric condition on boundary $ \partial \Omega $, that is, any $ \varphi \in C(\partial \Omega) $ that solves $ S_0 \varphi = 0 $ can only be trivial (see [15, Theorem 7.41]).
Rewrite (1.4) as
\begin{equation}
S_0 \varphi = \int_{\partial \Omega} G(x,y) \varphi (y) ds(y)
\end{equation}
where
\begin{equation}
G(x,y) := - \frac{1}{2\pi} \ln \vert x-y \vert + \frac{1}{2\pi} \frac{1}{\vert \partial \Omega \vert} \int_{\partial \Omega} \ln \vert x- z \vert d s(z) + \frac{1}{\vert \partial \Omega \vert}.
\end{equation}

\indent In the following, utilizing the technique in (SIE), we transform the modified Symm's integral equation into one-dimensional form.
Now the regular parameterization
\begin{equation*}
\partial \Omega:= \{ \gamma(t): t \in [0, 2\pi) \}
\end{equation*}
is twice continuously differentiable with $ \vert \dot \gamma (t) \vert >0, \ \forall t \in [0,2\pi] $. Inserting it into (1.5), the modified (SIE) takes the form
\begin{equation}
S_0 \varphi := \int^{2\pi}_0 G(t,s) \Psi(s) ds = g(t), \quad t \in [0,2\pi)
\end{equation}
 with the transformed kernel
 \begin{equation*}
G(t,s) :=
\end{equation*}
 \begin{equation*}
 - \frac{1}{2\pi} \ln \vert \gamma(t) - \gamma (s) \vert + \frac{1}{2\pi} \frac{1}{ \vert \partial \Omega \vert} \int^{2\pi}_0 \ln \vert \gamma(t)- \gamma(\sigma) \vert \vert \gamma'(\sigma) \vert d \sigma + \frac{1}{\vert \partial \Omega \vert},
\end{equation*}
 the transformed density $ \Psi(s) := \varphi(\gamma(s))\vert {\gamma'}(s)\vert $ and $ g(t) := f(\gamma(t)), \ s \in [0,2\pi] $.
\newline \indent The research on numerical analysis of modified Symm's integral equation for exterior problem of Laplace equation is few. It is indicated in [15, Example 13.23] that, for ellipse boundary curve, the (1.7) can be rewriten as
  \begin{equation}
  \frac{1}{2\pi} \int^{2\pi}_0 E(t,\tau) \Psi(\tau) d\tau = g(t) ,\quad t \in [0,2\pi]
  \end{equation}
  where
  \begin{equation*}
  E(t,\tau) =  (\ln (4 \sin^2 \frac{t-s}{2}) -2 ) + (\ln (4 \sin^2 \frac{t-s}{2}) ) K(t,\tau) +L(t,\tau)
  \end{equation*}
  and $ K $, $ L $ is infinitely differentiable and $ 2\pi $ periodic with respect to both variables such that $ K(t,t) = 0 $ for all $ 0 \leq t \leq 2\pi $. Then, the (1.8) can be well handled by Galerkin-Collocation and Petrov-Galerkin methods on trigonometric basis.
  Besides, there exist some work using modified boundary integral equation with Nystr\"{o}m method to handle exterior Neumann and Robin  problem of Laplace equation (See [16,21]). In this paper, we use Lemma 2.4 to determine that all (1.7) on analytic boundary curves can be transformed into form (1.8), and thus, we can extend the convergence and error analysis of Petrov-Galerkin, Galerkin-Collocation methods with trigonometric basis to the more general case.

\indent As to the arrangement of the rest contents. In section 2, we introduce necessary preliminaries, such as periodic Sobolev space, basic properties of modified Symm's integral operator.  In section 3, we analyze the convergence for three Petrov-Galerkin methods and Galerkin-Collocation method respectively. In section 4, we illustrate the numerical procedures and complete the numerical experiments, show the validness of convergence analysis. In section 5, we conclude the whole work of this paper.
\section{Preliminaries}
\subsection{Periodic Sobolev space $ H^r(0,2\pi) $, trace space $ H^k (\Gamma)$ and estimates}
Throughout this paper, we denote the $ 2\pi $ periodic Sobolev space of order $ r \in \mathrm{R} $ by $ H^r(0,2\pi) $ (refer to [13,15]). Notice that, for $ r > s $, the Sobolev space $ H^r(0,2\pi) $ is a dense subspace of $ H^s(0,2\pi) $. The inclusion operator from $ H^r(0,2\pi) $ into $ H^s(0,2\pi) $ is compact.
\newline \indent Let $ \Gamma $ be the boundary of a simply connected bounded domain $ D \subseteq \mathrm{R}^2 $ of class $ C^k, k \in \mathrm{N} $. With the aid of a regular and $ k $ times continuously differentiable $ 2\pi $ periodic paramater representation
\begin{equation*}
\Gamma = \{ z(t): t \in [0,2\pi ) \}
\end{equation*}
for $ 0 \leq p \leq k $ we can define the trace space $ H^p(\Gamma) $ as the space of all functions $ \varphi \in L^2(\Gamma) $ with the property that $ \varphi \circ z \in H^p(0,2\pi) $. By $ \varphi \circ z $, we denote the $ 2\pi $ periodic function given by $ (\varphi \circ z)(t) := \varphi(z(t)), t\in \mathrm{R} $. The scalar product and norm on $ H^p(\Gamma) $ are defined through the scalar product on $ H^p(0,2\pi) $ by
\begin{equation*}
(\varphi, \psi)_{H^p(\Gamma)}:= (\varphi \circ z, \psi \circ z)_{H^p(0,2\pi)}.
\end{equation*}
\begin{lemma}
Let $ P_n : L^2 ( 0, 2\pi ) \longrightarrow X_n \subset L^2(0, 2\pi ) $ be an orthogonal projection operator, where $ X_n = span \{ e^{ikt} \}^n_{k = -n} $. Then \begin{math} P_n \end{math} is given as follows£º
\begin{equation*}
(P_n x)(t) = \sum^n_{k = -n} a_k e^{ikt}, \quad x \in L^2(0,2\pi),
\end{equation*}
where
\begin{equation*}
a_k = \frac{1}{2\pi}\int^{2\pi}_0 x(s) \exp(-iks) ds, \quad k \in \mathrm{N},
  \end{equation*}
 are the Fourier coefficients of $ x $. Furthermore, the following estimate holds:
 \begin{equation*}
 \Vert x- P_n x \Vert_{H^s} \leq \frac{1}{ n^{r-s} } \Vert x \Vert_{H^r} \quad x \in  \ H^r(0,2\pi),
 \end{equation*}
 where $ r \geq s $.
\end{lemma}
\begin{proof}
See [13, Theorem A.43].
\end{proof}
\begin{lemma}
(Inverse inequality):
Let $ r \geq s $. Then there exists a $ c > 0 $ such that
\begin{equation*}
\Vert \psi_n \Vert_{H^r} \leq c n^{r-s} \Vert \psi_n \Vert_{H^s}, \quad \forall \ \psi_n \in X_n
\end{equation*}
for all $ n \in \mathrm{N} $.
\end{lemma}
\begin{proof}
See [13, Theorem 3.19].
\end{proof}
\subsection{Integral operator and regularity}
\begin{lemma}
Let $ r \in \mathrm{N} $ and $ k \in C^r([0,2\pi] \times [0,2\pi]) $ be $ 2\pi- $ periodic with respect to both variables. Then the integral operator $ K $, defined by
\begin{equation*}
(Kx)(t) := \int^{2\pi}_0 k(t,s) x(s) ds, \quad t \in (0,2\pi),
\end{equation*}
can be extended to a bounded operator from $ H^p(0,2\pi) $ into $ H^r(0,2\pi) $ for every $ -r \leq p \leq r $.
\end{lemma}
\begin{proof}
See [13, Theorem A.45].
\end{proof}
\begin{lemma}
Let $ \partial \Omega $ be the boundary of bounded simply connected domain $ \Omega \subseteq \mathrm{R}^2 $. If $ \partial \Omega $ is of class $ C^{m+1, \alpha} $ and $ \varphi $ of $ C^{m, \alpha} $ with $ m \in \mathrm{N} $ and $ 0 < \alpha < 1 $, then the interior single layer potential defined by $ \varphi $, that is,
\begin{equation*}
\Psi (x) := \int_{\partial \Omega} \varphi(y) \ln \vert x - y \vert d s(y), \quad x \in \overline{\Omega},
\end{equation*}
 is of class $ C^{m+1, \alpha} $ on $ \overline{\Omega} $.
\end{lemma}
\begin{proof}
See [6, Page 303]
\end{proof}
\subsection{Modified Symm's integral equation of the first kind}
Throughout this paper, we denote the modified Symm's integral operator in (1.7) by $ S_0 $.
\begin{equation}
(S_0 \Psi)(t) := \int^{2\pi}_0 G(t,s) \Psi(s) ds = g(t), t \in [0,2\pi)
\end{equation}
 with the transformed kernel
 \begin{equation*}
G(t,s) :=
\end{equation*}
 \begin{equation*}
 - \frac{1}{2\pi} \ln \vert \gamma(t) - \gamma (s) \vert + \frac{1}{2\pi} \frac{1}{ \vert \partial \Omega \vert} \int^{2\pi}_0 \ln \vert \gamma(t)- \gamma(\sigma) \vert \vert \gamma'(\sigma) \vert d \sigma + \frac{1}{\vert \partial \Omega \vert}.
\end{equation*}
Utilizing the common decomposition technique on kernel (see [13, Chapter 3.3]) in Symm's integral equation of the first kind, we split kernel $ G(t,s) $ into three parts:
\begin{equation}
G(t,s) = G_1(t,s) + G_2(t,s) + G_3(t),
\end{equation}
where
\begin{equation}
G_1(t,s) := - \frac{1}{4\pi} (\ln (4 \sin^2 \frac{t-s}{2}) -1 ) \quad (t \neq s)
\end{equation}
\begin{equation}
G_2(t,s) :=  - \frac{1}{2\pi} \ln \vert \gamma(t) - \gamma (s) \vert + \frac{1}{4\pi} (\ln (4 \sin^2 \frac{t-s}{2}) -1 ) \quad (t \neq s)
\end{equation}
\begin{equation}
G_3(t) := \frac{1}{2\pi} \frac{1}{ \vert \partial \Omega \vert} \int^{2\pi}_0 \ln \vert \gamma(t)- \gamma(\sigma) \vert \vert \gamma'(\sigma) \vert d \sigma + \frac{1}{\vert \partial \Omega \vert}.
\end{equation}
 We note that the logarithmic singularities at $ t = s $ in $ G(t,s) $ is separated to $ G_1 $, and $ G_1 $ corresponds to the regular representation of disc with center $ 0 $ and radius $ a = e^{-\frac{1}{2}} $, that is,
 \begin{equation*}
  \gamma_a (s)= e^{-\frac{1}{2}} ( \cos s, \sin s ), \ s \in [0,2\pi).
  \end{equation*}
  The second part $ G_2 $  has a analytic continuation onto $ [0,2\pi] \times [0,2\pi] $ (See [13, Page 84]) since $ \gamma $ is analytic. The third part
  \begin{equation*}
   G_3(t) = - \frac{1}{\vert \partial \Omega \vert} h(t) + \frac{1}{\vert \partial \Omega \vert}, \quad t \in [0,2\pi],
   \end{equation*}
   where
   \begin{equation*}
h(t) = - \frac{1}{2\pi}  \int^{2\pi}_0 \ln \vert \gamma(t)- \gamma(\sigma) \vert \vert \gamma'(\sigma) \vert d \sigma
\end{equation*}
is the single layer potential of constant function $ 1 $ on $ \partial \Omega $ which is analytic. By Lemma 2.4, $ G_3(t) \in C^{\infty}[0,2\pi] $
\newline \indent Now we define integral operators respectively as
\begin{equation}
(S_1 \Psi)(t) :=  \int^{2\pi}_0 G_1(t,s) \Psi (s)  ds
\end{equation}
\begin{equation}
(S_2 \Psi)(t) :=  \int^{2\pi}_0 ( G_2(t,s) + G_3 (t)) \Psi(s) ds.
\end{equation}
\begin{equation}
(K_2 \Psi)(t) :=  \int^{2\pi}_0 ( G_2(t,s) + G_3 (s)) \Psi(s) ds.
\end{equation}
\begin{equation}
(K \Psi)(t) :=  \int^{2\pi}_0 ( G_1(t,s) + G_2(t,s) + G_3 (s)) \Psi(s) ds.
\end{equation}
\begin{equation}
S_0 = S_1 + S_2, \quad K = S_1 + K_2.
\end{equation}
\begin{lemma}
It holds that
\begin{equation*}
\frac{1}{2\pi} \int^{2\pi}_0 e^{ins} \ln (4 \sin^2 \frac{s}{2}) ds =
\left\{
\begin{array}{rcl}
-\frac{1}{\vert n \vert}, & &  {n \in \mathrm{Z}, n \neq 0},  \\
0,  & &  {n = 0}.
\end{array}
\right.
\end{equation*}
This gives that the functions
\begin{equation*}
\hat \psi_n (t) := e^{int}, \quad t \in [0,2\pi], \ n \in \mathrm{Z},
\end{equation*}
are eigenfunctions of $ S_1 $:
\begin{equation*}
S_1 \hat \psi_n = \frac{1}{2 \vert n \vert} \hat \psi_n \quad \textrm{for} \ n \neq 0 \ \textrm{and}
\end{equation*}
\begin{equation*}
S_1 \hat \psi_0 = \frac{1}{2} \hat \psi_0.
\end{equation*}
\end{lemma}
\begin{proof}
See [13, Theorem 3.17]
\end{proof}
\begin{lemma}
Let $ \Omega \subseteq \mathrm{R}^2 $ be a simply connected bounded domain with $ \partial \Omega $ be its boundary analytic. Then
 \newline (a) $ S_0 $ is compact in $ L^2(0,2\pi) $ and $ K = S^*_0 $ when we see $ K, S_0 $ both as operator on $ L^2(0,2\pi) $.
 \newline (b) The operator $ S_1 $ is bounded injective from $ H^{s-1} (0,2\pi) $ onto $ H^s(0,2\pi) $ with bounded inverses for every $ s \in \mathrm{R} $, the same assertion also holds for $ S_0, K $.
 \newline (c) The operator $ S_1 $ is coercive from $ H^{-\frac{1}{2}} (0,2\pi) $ into $ H^{\frac{1}{2}} (0,2\pi) $.
 \newline (d) The operator $ S_2, K_2 $ is compact from $ H^{s-1} (0,2\pi) $ into $ H^s(0,2\pi) $ for every $ s \in \mathrm{R} $.
\end{lemma}
\begin{proof}
See [13, Theorem A.33, Theorem 3.18] for (a),(c) and the former part of (b).
\newline \indent Following the main idea in [13, theorem 3.18], we prove (d) and the latter part of (b). Since the $ G_2(t,s) + G_3(t,s) $ has a $ C^\infty $ continuation on $ [0,2\pi] \times [0,2\pi] $, by Lemma 2.3, $ S_2 $ defines a bounded operator from $ H^{s-1}(0,2\pi) $ to $ H^p (0,2\pi) $ with $ s < p $. Composing with compact embedding $ H^p (0,2\pi) \subset \subset H^s(0,2\pi) $, (d) follows.
\newline \indent For the latter part of (b), we see $ S_0 = S_1 ( I + S^{-1}_1 S_2) $ and $ K = S_1 ( I + S^{-1}_1 K_2) $. Notice that, for $ s \in \mathrm{R} $, by (d) and former part of (b), $ S^{-1}_1 S_2 $ and $ S^{-1}_1 K_2 $ are all compact operators on $ H^{s-1}(0,2\pi) $. Now by Riesz theorem (See [13, Theorem A.34 (b)]), if we prove that $ ( I + S^{-1}_1 S_2) $ and $ ( I + S^{-1}_1 K_2) $ are all injective on $ H^{s-1} (0,2\pi) $, that is, $ S_0 $ and $ K $ are all injective from $ H^{s-1} (0,2\pi) $ to $ H^s(0,2\pi) $, then we prove that $ ( I + S^{-1}_1 S_2) $ and $ ( I + S^{-1}_1 K_2) $ are all surjective on $ H^{s-1} (0,2\pi) $ with bounded inverses, that is, $ S_0 $ and $ K $ are all surjective from $ H^{s-1} (0,2\pi) $ to $ H^s(0,2\pi) $ with bounded inverses.
\newline \indent Now it is sufficient to prove the injectivity of $ S_0, K $ from $ H^{s-1} (0,2\pi) $ to $ H^s (0,2\pi) $ with $ s \in \mathrm{R} $. Let $ \Psi \in H^{s-1}(0,2\pi) $ with $ S_0 \Psi = 0 $. From $ S_1 \Psi = - S_2 \Psi $ and the mapping properties (Lemma 2.3) of $ S_2 $, we know $ S_1 \Psi \in H^p (0,2\pi), \forall p \in \mathrm{R} $ and thus, $ \Psi \in H^1 (0,2\pi) $. This implies that $ \Psi $ is continuous and the transformed function $ \varphi (\gamma(t)) = \frac{\Psi(t)}{\vert \gamma'(t) \vert} $ satifies (1.4) for $ g = 0 $. The injectivity of $ S_0 $ on $ C(\partial \Omega) $ gives $ \varphi = 0 $.
\newline \indent Notice that when $ K,S_0 $ are defined on $ L^2(0,2\pi) $, $ \mathcal{N}(K) = \mathcal{N}(S^*_0) = \mathcal{R}(S_0)^\perp = 0 $.
Let $ \Psi \in H^{s-1}(0,2\pi) $ with $ K \Psi = 0 $. From $ S_1 \Psi = - K_2 \Psi $ and the mapping properties (Lemma 2.3) of $ K_2 $, we know $ S_1 \Psi \in H^2 (0,2\pi) $ and thus, $ \Psi \in H^1 (0,2\pi) \subseteq L^2 (0,2\pi) $. Thus, $ \Psi = 0 $.
\end{proof}

\section{Convergence analysis for Petrov-Galerkin methods}
Let $ \Psi^\dagger \in H^r(0,2\pi) $ be the unique solution of (1.7); that is,
\begin{equation*}
(S_0 \Psi^\dagger)(t): = \int^{2\pi}_0 G(t,s) \Psi^\dagger(s) ds = g(t) := f(\gamma(t)),
\end{equation*}
for $ t \in [0,2\pi] $ and some $ g \in H^{r+1} (0,2\pi) $ for $ r \geq 0 $. Let $ g^\delta \in L^2(0,2\pi) $ with $ \Vert g^\delta - g \Vert_{L^2} \leq \delta $ and $ X_n $ defined by
\begin{equation}
X_n = span \{e^{ikt} \}^{n}_{k=-n},
\end{equation}
\subsection{Least squares method}
Let $ \Psi^\delta_n $ be the least squares solution of (1.7); that is,
\begin{equation*}
(S_0 \Psi^\delta_n ,  S_0 \psi_n) = (g^\delta, S_0 \psi_n ) \quad \textrm{for all} \ \psi_n \in X_n.
\end{equation*}
Then there exists $ c> 0 $ with
\begin{equation*}
\Vert \Psi^\dagger - \Psi^\delta_n \Vert_{L^2} \leq C ( n \delta + \frac{1}{n^r} \Vert x \Vert_{H^r}).
\end{equation*}

\begin{lemma}
(Stability estimate):
There exists a $ c>0 $, independent of $ n $, such that
\begin{equation}
\Vert \Psi_n \Vert_{L^2} \leq C n \Vert S_0 \Psi_n \Vert_{L^2} \ \textrm{for all}  \ \Psi_n \in X_n.
\end{equation}
The assertion also holds for the adjoint operator $ S^*_0 $, that is,
\begin{equation}
\Vert \Psi_n \Vert_{L^2} \leq C n \Vert S^*_0 \Psi_n \Vert_{L^2} \ \textrm{for all}  \ \Psi_n \in X_n.
\end{equation}
\end{lemma}
\begin{proof}
Similar to [13, Lemma 3.19], for $ \Psi_n = \sum^{n}_{k=-n} a_k e^{ikt} \in X_n $,
\begin{equation*}
\Vert S_1 \Psi_n \Vert^2_{L^2} = \frac{\pi}{2} [ \vert a_0 \vert^2 + \sum_{\vert j \vert \leq n, j \neq 0} \frac{1}{j^2} \vert a_j \vert^2 ] \geq \frac{1}{n^2} \Vert \Psi_n \Vert^2.
\end{equation*}
which proves estimate (3.2) for $ S_1 $. The estimate for $ S_0 $ follows from the observation that $ S_0 = (S_0 S^{-1}_1) S_1 $ and that $ (S_0 S^{-1}_1) $ is bounded with bounded inverse in $ L^2(0,2\pi) $ by Lemma 2.6 (b). As to the adjoint case, $ S^*_0 \Psi_n  =  K \Psi_n , \ \forall \Psi_n \in X_n $, again using above observation, (3.3) follows.
\end{proof}
\begin{proof}
\begin{equation*}
\min_{z_n \in X_n}\Vert x - z_n \Vert \leq \Vert x - P_n x  \Vert \leq 2 \Vert x \Vert
\end{equation*}
\begin{equation*}
 \sigma_n(S_0) := \max \{ \Vert \psi_n \Vert_{L^2}: \psi_n \in X_n, \ \Vert S_0 \psi_n \Vert_{L^2} = 1  \} \leq C n \quad (\textrm{by} \ (3.2))
\end{equation*}
\begin{equation*}
 \Vert S_0 ( x - P_n x) \Vert_{L^2} \leq \Vert S_0 \Vert_{H^{-1} \to L^2} \Vert x - P_n x \Vert_{H^{-1}}  \leq \frac{C}{n} \Vert x \Vert_{L^2} \quad \textrm{for all} \ x \in L^2(0,2\pi).
\end{equation*}
Then it yields that
\begin{equation*}
\min_{z_n \in X_n} \{ \Vert x - z_n \Vert_{L^2} + \sigma_n \Vert S_0 (x-z_n) \Vert_{L^2} \} \leq    C \Vert x \Vert_{L^2}, \quad \textrm{for all} \ x \in L^2(0,2\pi).
\end{equation*}
Application of [13, Theorem 3.10] yields that
\begin{equation*}
\Vert \Psi^\dagger - \Psi^\delta_n \Vert_{L^2} \leq C (n \delta +  \min\{  \Vert \Psi^\dagger - z_n \Vert_{L^2}: z_n \in X_n \})
\end{equation*}
together with Lemma 2.1, the desired result yields.
\end{proof}

\subsection{Dual least squares method}
Let $ \Psi^\delta_n = S_0 \tilde \psi^\delta_n  $ be with $ \tilde \psi^\delta_n \in X_n $ be the dual least squares solution of (1.7); that is, $ \tilde \psi^\delta_n $ solves
\begin{equation*}
(S_0 \tilde \psi^\delta_n ,  S_0 \psi_n) = (g^\delta, \psi_n ) \quad \textrm{for all} \ \psi_n \in X_n.
\end{equation*}
Then there exists $ c> 0 $ with
\begin{equation*}
\Vert \Psi^\dagger - \Psi^\delta_n \Vert_{L^2} \leq C ( n \delta + \min\{  \Vert \Psi - z_n \Vert_{L^2}: z_n \in S^*_0 (X_n) \}).
\end{equation*}
\begin{proof}
Notice that
\begin{equation*}
 \sigma'_n(S_0) := \max \{ \Vert \psi_n \Vert_{L^2}: \psi_n \in X_n, \ \Vert S^*_0 \psi_n \Vert_{L^2} = 1  \} \leq C n \quad (\textrm{by} \ (3.3))
\end{equation*}
Then, application of [13, Theorem 3.11] yields the desired result.
\end{proof}
\subsection{Bubnov-Galerkin method}
Let $ \Psi^\delta_n  \in X_n $ be the Bubnov-Galerkin solution; that is, the solution of
\begin{equation*}
(S_0 \Psi^\delta_n ,  S_0 \psi_n) = (g^\delta, \psi_n ) \quad \textrm{for all} \ \psi_n \in X_n.
\end{equation*}
Then there exists $ c> 0 $ with
\begin{equation*}
\Vert \Psi^\dagger - \Psi^\delta_n \Vert_{L^2} \leq C ( n \delta + \frac{1}{n^r} \Vert x \Vert_{H^r}).
\end{equation*}
\begin{proof}
Following [13, Theorem 3.20], set $ V = H^{\frac{1}{2}}(0,2\pi) $ and $ V^* = H^{-\frac{1}{2}} (0,2\pi) $, with Lemma 2.6 (c) and (d) of $ s = \frac{1}{2} $, we know $ S_0: H^{- \frac{1}{2}} (0,2\pi) \to H^{\frac{1}{2}} (0,2\pi) $ satisfies G\"{a}rding inequality with $ - S_2 $ defined in (2.7). With application of Lemma 2.2 of $ r = 0, s = - \frac{1}{2} $, we have
\begin{equation*}
\rho_n := \max \{ \Vert \psi_n \Vert_{L^2}: \psi_n \in X_n, \Vert \psi_n \Vert_{H^{-\frac{1}{2}}} = 1 \} \leq c \sqrt{n}.
\end{equation*}
By Lemma 2.1, we have
\begin{equation*}
\Vert u - P_n u \Vert_{H^{-\frac{1}{2}}} \leq c \sqrt{n} \Vert u \Vert_{L^2} \quad {\textrm{for all}} \ u \in L^2(0,2\pi)
\end{equation*}
Thus, by [13, Theorem 3.14], we have
\begin{equation*}
 \Vert \Psi^\dagger - \Psi^\delta_n \Vert_{L^2} \leq  c (n \delta + \Vert (I- P_n ) \Psi^\dagger \Vert_{L^2}).
\end{equation*}
Together with Lemma 2.1, the desired result yields.
\end{proof}

\section{Analysis for Galerkin-Collocation method}
Define collocation points by
\begin{equation*}
t_k := k \frac{\pi}{n}, \quad k = 0,1,\cdots, 2n-1.
\end{equation*}
The collocation equation take the form
\begin{equation}
\int^{2\pi}_0 G(t_k,s) \Psi_n(s) ds = g(t_k),  \quad k = 0,1,\cdots, 2n-1.
\end{equation}
 with $ \Psi_n \in X_n := \textrm{span} \{ e^{i j t }\}^{n-1}_{j = -n} $ and
 \begin{equation*}
G(t_k,s) :=
\end{equation*}
 \begin{equation*}
 - \frac{1}{2\pi} \ln \vert \gamma(t_k) - \gamma (s) \vert + \frac{1}{2\pi} \frac{1}{ \vert \partial \Omega \vert} \int^{2\pi}_0 \ln \vert \gamma(t_k)- \gamma(\sigma) \vert \vert \gamma'(\sigma) \vert d \sigma + \frac{1}{\vert \partial \Omega \vert}.
\end{equation*}
Using the decomposition technique that $ S_0 = S_1 + S_2 $, completely similar to [13, Theorem 3.27], we can obtain that
\begin{theorem}
The collocation method is convergence for (1.8) with analytic boundary: that is, the solution $ \Psi_n \in X_n $ converges to the solution $ \Psi^\dagger \in L^2(0,2\pi) $ of (1.8) in $ L^2(0,2\pi) $.
\newline Let the RHS of (4.1) be replaced by $ \beta^\delta \in \mathrm{C}^{2n} $ with
\begin{equation*}
\sum^{2n-1}_{k=0} \vert \beta^\delta_k - g(t_k) \vert^2 \leq \delta^2 .
\end{equation*}
Let $ \alpha^\delta \in \mathrm{C}^{2n} $ be the solution of $ A\alpha^\delta = \beta^\delta $, where $ A_{kj} = S_0(\hat x_j) (t_k) $. Then the following error estimate holds:
\begin{equation*}
\Vert \Psi^\delta_n - \Psi^\dagger \Vert_{L^2} \leq c[ \sqrt{n} \delta + \min\{ \Vert \Psi^\dagger - \psi_n \Vert_{L^2} : \psi_n \in X_n \} ]
\end{equation*}
If $ \Psi^\dagger \in H^r (0,2\pi) $ for some $ r > 0 $, then
\begin{equation*}
\Vert \Psi^\delta_n - \Psi^\dagger \Vert_{L^2} \leq c [\sqrt{n} \delta + \frac{1}{n^r}\Vert \Psi^\dagger \Vert_{H^r} ].
\end{equation*}

\end{theorem}

\section{Numerical experiments}
All experiments are performed in Intel(R) Core(TM) i7-7500U CPU @2.70GHZ 2.90 GHZ Matlab R 2017a.  Here we illustrate the computation procedure in details (refer to [13, Section 3.5] or [14])¡£
\newline \indent The previous paper is mainly on the numerical analysis of PG and GC methods on Fourier basis. Since the PG and GC methods on Fourier basis is equivalent to these on trigonometric interpolation basis (See [14,15]), we implement PG and GC methods on trigonometric interpolation basis.
\newline \indent We first introduce the trigonometric interpolation basis $ \{ L_j (t) \}^{2n-1}_{j=0} $,
\begin{equation*}
L_j(t) = \frac{1}{2n} (1 + 2 \sum^{n-1}_{k=1} \cos k(t - t_j) + \cos n(t-t_j)), \quad j = 0,1,\cdots,2n-1.
\end{equation*}
Notice that
\begin{equation*}
L_j(t_k) = 1, \quad \textrm{if} \ k = j,
\end{equation*}
\begin{equation*}
 L_j(t_k) = 0, \quad \textrm{if} \  k \neq j.
\end{equation*}
 and corresponding trigonometric interpolation operator $ \Pi_n \Psi := \sum^{2n-1}_{j=0} \Psi(t_j) L_{j}(t)  $, notice that $ \Pi_n \Psi \in X_n $, where
\begin{equation*}
X_n = \{ \sum^n_{j=0} a_j \cos(jt) + \sum^{n-1}_{j = 1} b_j \sin(jt): \quad a_j,b_j \in \mathrm{R} \}
\end{equation*}
is $ 2n $ dimensional subspace.
\newline \indent Without introduction of extra techniques of numerical approximation, one can not obtain a proper implementation of PG methods for the logarithmic singularity in modified Symm's integral operator.
In the following, we use the composite trapezial quadature formula and trigonometric interpolation to eliminate singularities in $ S_0 \Psi $. Notice that, given exact solution $ \Psi^\dagger $, if we can obtain a approximation to $ S_0 \Psi^\dagger $ with high precision and no singularity, then we can similarly implement $ S_0 L_j, \ j = 0,\cdots,2n-1 $. The latter is the key point to form corresponding matrix system.
\newline \indent Set
\begin{equation}
(S_0 \Psi)(t) := \int^{2\pi}_0 G(t,s) \Psi(s) ds = g(t), \quad t \in [0,2\pi)
\end{equation}
where
 \begin{equation*}
G(t,s) :=
 - \frac{1}{\pi} \ln \vert \gamma(t) - \gamma (s) \vert + \frac{1}{\pi} \frac{1}{ \vert \partial \Omega \vert} \int^{2\pi}_0 \ln \vert \gamma(t)- \gamma(\sigma) \vert \vert \gamma'(\sigma) \vert d \sigma + \frac{2}{\vert \partial \Omega \vert}.
\end{equation*}

 \indent Do decomposition, rewrite $ S_0 \Psi = g $ as
\begin{equation*}
(S_0 \Psi) (t) = S_K \Psi  +    2G_3(t) ( \int^{2\pi}_0 \Psi (s) ds );
\end{equation*}
\begin{equation*}
S_K(\Psi) = - \frac{1}{\pi} \int^{2\pi}_0 \ln \vert \gamma(t) -\gamma(s) \vert \Psi(s) ds, \quad t \in [0,2\pi],
\end{equation*}
\begin{equation*}
2G_3(t) = \frac{1}{\pi} \frac{1}{ \vert \partial \Omega \vert} \int^{2\pi}_0 \ln \vert \gamma(t)- \gamma(\sigma) \vert \vert \gamma'(\sigma) \vert d \sigma + \frac{2}{\vert \partial \Omega \vert}
\end{equation*}
Notice that if we can implement an good approximation to $ S_K \Psi $, then we can similarly implement $  \frac{1}{\pi} \frac{1}{ \vert \partial \Omega \vert} \int^{2\pi}_0 \ln \vert \gamma(t)- \gamma(\sigma) \vert \vert \gamma'(\sigma) \vert d \sigma $ as $ - \frac{1}{ \vert \partial \Omega \vert} S_K \vert \gamma'(\sigma) \vert $. Thus the difficulties in numerical implementation for $ S_0 \Psi $ are overcome. Now we introduce the procedure to implement $ S_K \Psi $.
\newline \indent Rewrite $ S_K \Psi $ as
\begin{equation*}
S_K \Psi = - \frac{1}{2\pi} \int^{2\pi}_0 \Psi(s) \ln (4 \sin^2 (\frac{t-s}{2})) ds + \int^{2\pi}_0 \Psi (s) k(t,s) ds
\end{equation*}
where $ t \in[0,2\pi] $ and  analytic function
\begin{equation*}
 k (t,s) = -\frac{1}{2\pi} \ln \frac{\vert \gamma(t) - \gamma(s) \vert^2}{4 \sin^2 (\frac{t-s}{2})}, \quad t \neq s ,
\end{equation*}
\begin{equation*}
k(t,t) = -\frac{1}{\pi} \ln \vert \gamma' (t) \vert, \quad 0 \leq t \leq 2\pi.
\end{equation*}
 Using the composite trapezial formula for periodic function, set $ t_j = j \frac{\pi}{n}, \ j = 0,\cdots, 2n-1  $.
\begin{equation*}
 \int^{2\pi}_0 \Psi (s) k  (t,s) ds \approx \frac{\pi}{n} \sum^{2n-1}_{j=0} k (t,t_j) \Psi(t_j) , \quad 0 \leq t \leq 2\pi.
\end{equation*}

\indent As to the approximation to the weakly singular part, using trigonometric interpolation, we have
\begin{equation*}
-\frac{1}{2\pi} \int^{2\pi}_0 \Psi(s) \ln(4 \sin^2 \frac{t-s}{2}) ds \approx - \frac{1}{2\pi} \int^{2\pi}_0 (\Pi_n \Psi) (s) \ln(4 \sin^2 \frac{t-s}{2}) ds
\end{equation*}
\begin{equation*}
 = \sum^{2n-1}_{j = 0 }\Psi(t_j) R_j (t), \quad 0 \leq t \leq 2\pi.
\end{equation*}
where, for $ j = 0,\cdots, 2n-1 $,
\begin{equation*}
R_j(t) = - \frac{1}{2\pi}  \int^{2\pi}_0 L_j(t) \ln(4 \sin^2 \frac{t-s}{2}) ds
= \frac{1}{n} \{  \frac{1}{2n} \cos n(t - t_j) + \sum^{n-1}_{m=1} \frac{1}{m} \cos m(t-t_j)  \}
\end{equation*}
Thus, we obtain an approximation formula for $ S_K \Psi $
\begin{equation*}
(S^{(n)}_K \Psi) (t) := \sum^{2n-1}_{j=0} \Psi(t_j) [ R_j(t) + \frac{\pi}{n} k (t,t_j) ], \quad 0 \leq t \leq 2 \pi.
\end{equation*}
Notice that $ S^{(n)}_K \Psi $ converges to $ S_K \Psi $ uniformly for all $ 2\pi $ periodic continuous function $ \Psi $. Furthermore, if $ \Psi $ is analytic, then the error $ \Vert S^{(n)}_K \Psi - S_K \Psi \Vert_{\infty} $ exponentially decreasing. (See [13, Page 105]).
\newline \indent Let $ n $ be large enough, $ S^{(n)}_K \Psi^\dagger $ is a precise approximation to $  S_K \Psi^\dagger $. However, $ S^{(n)}_K \Psi^\dagger $ still possesses singularity in $ k(t,t_j) $, we add an interpolation step for $ S^{(n)}_K \Psi^\dagger $ to eliminate the singularity, that is,
\begin{equation*}
(S^{(n)}_{TK} \Psi^\dagger) (t) := \sum^{2n-1}_{j=0} \Psi^\dagger (t_j)  R_j(t) + \Pi_n( \sum^{2n-1}_{j=0} \Psi^\dagger (t_j) \frac{\pi}{n} k (t,t_j)) .
\end{equation*}
It can be known that, if $ n $ is sufficiently large, then $ S^{(n)}_{TK} \Psi^\dagger \approx S^{(n)}_K \Psi^\dagger \approx S_K \Psi^\dagger $. For details in error analysis, see [13, Chapter 3] and [15]. In the proceeding content, we will uniformly use
\begin{equation*}
S^{(10)}_{TK} \Psi^\dagger + (  - \frac{1}{ \vert \partial \Omega \vert} S^{(10)}_{TK} \vert \gamma'(\sigma) \vert + \frac{2}{\vert \partial \Omega \vert} ) ( \int^{2\pi}_0 \Psi^\dagger (s) ds )
\end{equation*}
 to replace $ S_0 \Psi^\dagger $.
\newline \indent With above preparation, we can form corresponding matrix system and RHS with no singularity. Before performing the numerical experiments. We introduce the following indexes
\begin{equation*}
g = S^{(10)}_{TK} \Psi^\dagger + (  - \frac{1}{ \vert \partial \Omega \vert} S^{(10)}_{TK} \vert \gamma'(\sigma) \vert + \frac{2}{\vert \partial \Omega \vert} ) ( \int^{2\pi}_0 \Psi^\dagger (s) ds ),
 \end{equation*}
 \begin{equation*}
 g^\delta = g + \delta \frac{\sin 6 t}{\sqrt{\pi}}, \quad \Vert g^\delta - g \Vert_{L^2} =\delta;
\end{equation*}
\begin{equation*}
r := \Vert \Psi^{\delta, \dagger}_n - \Psi^\dagger \Vert_{L^2}
\end{equation*}
where $ \Psi^{\delta, \dagger}_n $ is the least squares, dual least squares, Bubnov-Galerkin solution corresponding to disturbed RHS $ g^\delta $. As to the experiments for exterior Dirichlet problem of Laplace equation, we introduce
\begin{equation*}
u_\infty = u^\dagger (x), \quad x = 10^n d, \ n = 2,4,6,8,10
\end{equation*}
and
\begin{equation*}
u^n_{\infty} = u_n (x), \quad x = 10^n d,  \ n = 2,4,6,8,10
\end{equation*}
\begin{equation*}
Err := \max_{1 \leq i,j \leq 20} \vert u_n (x_{i,j}) - u^\dagger (x_{i,j}) \vert
\end{equation*}
where
\begin{equation*}
d_1 = (1,0)
\end{equation*}
\begin{equation*}
u^\dagger = S_0 \Psi^\dagger, \quad u_n = S_0 \Psi^{\dagger,0}_n
\end{equation*}
\begin{equation*}
x_{i,j} = (0.1 + i, 1.1+j ), \quad 1 \leq i,j \leq 20.
\end{equation*}
We  introduce the $ u_\infty $, $ u^n_\infty $ and $ \vert u_\infty - u^n_\infty  \vert $ to illustrate the behaviour of solution and approximate solution at infinity, that is,  if the solutions and numerical solutions of exterior Dirichlet problem satisfies the boundedness and uniform convergence as stated in theory. Besides, we use $ Err $ to illustrate the whole precision of numerical solutions in different numerical methods.

\begin{example}
 Let the boundary $ \partial \Omega $ of $ \Omega $ be parameterized by $ \gamma(t) = (\cos t , 2 \sin t  ), \ t \in [0,2\pi]  $. We know that $ \Omega $ is bounded and simply connected with non-zero tangent vector in every point of $ \partial \Omega $. The $ \partial \Omega $ is described as

\begin{figure}[htbp]
\centering
 \includegraphics[width=0.8\textwidth]{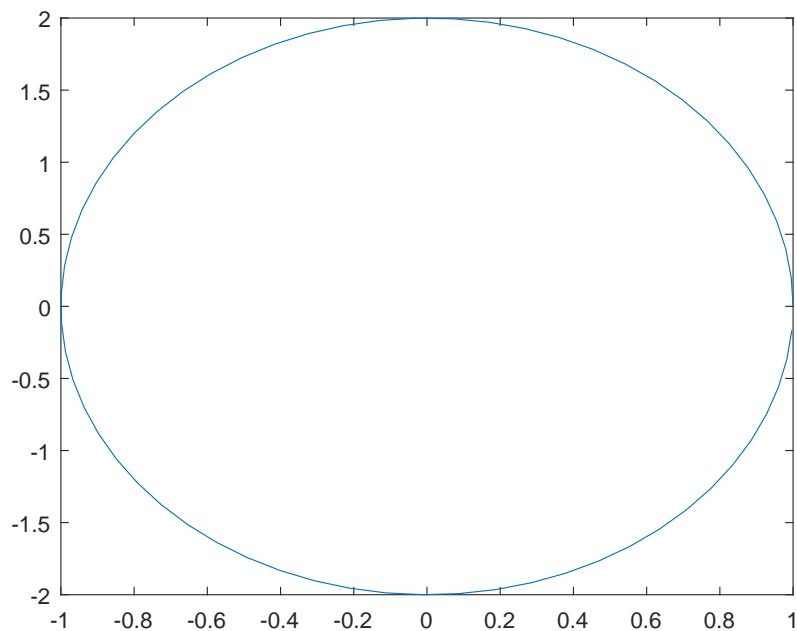}
\caption{$ \Omega $ for Example 5.1}
\end{figure}

Now
\begin{equation*}
k(t,s) =  -\frac{1}{2\pi} \ln \frac{  ( \cos t - \cos s )^2 + 4( \sin t  - \sin s )^2   }{4 \sin^2 (\frac{t-s}{2})},  \quad t \neq s ,
 \end{equation*}
 \begin{equation*}
k (t,t) = -\frac{1}{2\pi} \ln (\sin^2 s  + 4\cos^2 s   ), \quad 0 \leq t \leq 2\pi,
\end{equation*}

 \begin{equation*}
 2G_3(t) = - \frac{1}{ \vert \partial \Omega \vert} (-\frac{1}{2\pi} \int^{2\pi}_0 \ln (4 \sin^2 (\frac{t-\sigma}{2}))
 \sqrt{ \sin^2 \sigma + 4\cos^2 \sigma  } d\sigma
 \end{equation*}
 \begin{equation*}
  -\frac{1}{2\pi} \int^{2\pi}_0 \ln \frac{( \cos t - \cos \sigma )^2 + 4( \sin t  - \sin \sigma )^2   }{4 \sin^2 (\frac{t-\sigma}{2})}
    \sqrt{ \sin^2 \sigma  + 4 \cos^2 \sigma  }  d \sigma)
    \end{equation*}
    \begin{equation*}
  + \frac{2}{\vert \partial \Omega \vert}, \quad t \in [0,2\pi].
\end{equation*}
\end{example}

The following experiments are for modified Symm's integral equation of the first kind: Table 1-6.
\begin{table}
{\begin{tabular}{ccccccccc}
\hline
 & & $ n $ & $ 2 $ & $ 4 $ & $ 6 $ & $ 8 $ & 10 & 12  \\
\cline{2-9}
 & $ \delta = 0 $  & $ r $ & $3.6082 $ & $ 0.3327 $ & $ 0.0161 $  & $ 4.7465 e^{-4} $   & $ 2.0573 e^{-5} $  & $ 2.0573 e^{-5} $   \\
 \cline{2-9}
& $ \delta = 0.001 $  & $ r $ & $3.6082 $ & $ 0.3327 $ & $ 0.0161 $  & $ 0.0060 $   & $ 0.0060 $  & $ 0.0060 $  \\
\cline{2-9}
 & $ \delta = 0.01 $ & $ r $ & $3.6082 $ & $ 0.3327 $ & $ 0.0161 $  & $ 0.0601 $   & $ 0.0601 $  & $ 0.0601 $ \\
 \cline{2-9}
 & $ \delta = 0.1 $  & $ r $ & $3.6082 $ & $ 0.3327 $ & $ 0.0161 $  & $ 0.6008 $   & $ 0.6008 $  & $ 0.6008 $ \\
\hline
\end{tabular}}
\caption{Least squares method for Example 5.1 }
\end{table}

\begin{table}
{\begin{tabular}{ccccccccc}
\hline
 & & $ n $ & $ 2 $ & $ 4 $ & $ 6 $ & $ 8 $ & 10 & 12  \\
\cline{2-9}
 & $ \delta = 0 $  & $ r $ & $3.6082 $ & $ 0.3327 $ & $ 0.0161 $  & $ 4.7465 e^{-4} $   & $ 2.0573 e^{-5} $  & $ 2.0635 e^{-5} $   \\
\hline
\end{tabular}}
\caption{Dual least squares method for Example 5.1. }
\end{table}

\begin{table}
{\begin{tabular}{ccccccccc}
\hline
 & & $ n $ & $ 2 $ & $ 4 $ & $ 6 $ & $ 8 $ & 10 & 12  \\
\cline{2-9}
 & $ \delta = 0 $  & $ r $ & $3.6082 $ & $ 0.3327 $ & $ 0.0161 $  & $ 4.7465 e^{-4} $   & $ 2.0573 e^{-5} $  & $ 2.0573 e^{-5} $   \\
 \cline{2-9}
& $ \delta = 0.001 $  & $ r $ & $3.6082 $ & $ 0.3327 $ & $ 0.0161 $  & $ 0.0060 $   & $ 0.0060 $  & $ 0.0060 $  \\
\cline{2-9}
 & $ \delta = 0.01 $ & $ r $ & $3.6082 $ & $ 0.3327 $ & $ 0.0161 $  & $ 0.0601 $   & $ 0.0601 $  & $ 0.0601 $ \\
 \cline{2-9}
 & $ \delta = 0.1 $  & $ r $ & $3.6082 $ & $ 0.3327 $ & $ 0.0161 $  & $ 0.6008 $   & $ 0.6008 $  & $ 0.6008 $ \\
\hline
\end{tabular}}
\caption{Bubnov-Galerkin method for Example 5.1. }
\end{table}
\begin{table}
{\begin{tabular}{ccccccccc}
\hline
 & & $ n $ & $ 2 $ & $ 4 $ & $ 6 $ & $ 8 $ & 10 & 12  \\
\cline{2-9}
 & $ \delta = 0 $  & $ r $ & $3.7405 $ & $ 0.3828 $ & $ 0.0197 $  & $ 6.0457 e^{-4} $   & $ 2.0573 e^{-5} $  & $ 2.0573 e^{-5} $   \\
 \cline{2-9}
& $ \delta = 0.001 $  & $ r $ & $3.7405 $ & $ 0.3828 $ & $ 0.0197 $  & $ 0.0060 $   & $ 0.0060 $  & $ 0.0060 $  \\
\cline{2-9}
 & $ \delta = 0.01 $ & $ r $ & $3.7405 $ & $ 0.3834 $ & $ 0.0197 $  & $ 0.0601 $   & $ 0.0601 $  & $ 0.0601 $ \\
 \cline{2-9}
 & $ \delta = 0.1 $  & $ r $ & $3.7405 $ & $ 0.4440 $ & $ 0.0197 $  & $ 0.6008 $   & $ 0.6008 $  & $ 0.6008 $ \\
\hline
\end{tabular}}
\caption{Galerkin-collocation method for Example 5.1. }
\end{table}
\begin{table}
{\begin{tabular}{ccccccc}
\hline
 & &  & LS & DLS & BG & GC  \\
\cline{2-7}
  $ \delta = 0 $  & $ n = 20 $ & $ r $ & $ 2.1222 e^{-5} $ & $ 294.4241 $ & $ 2.0973 e^{-5} $  & $ 2.0573 e^{-5} $     \\
\hline
\end{tabular}}
\caption{Additional experiments for PG and GC methods on Example 5.1 with $ n = 20 $ }
\end{table}

\begin{table}
{\begin{tabular}{ccccccccc}
\hline
& & $ n $ & $ 2 $ & $ 4 $ & $ 6 $ & $ 8 $& 10 & 12 \\
\cline{2-9}
 LS &   & $ t $ & 17.2271 & 59.4760 & 135.8122 & 237.7139 & 382.7361 & 551.4632 \\
\cline{2-9}
 DLS &  & $ t $ & 11.4329 &  33.0316 & 72.4827 & 126.1410 & 191.6831 & 323.8654 \\
\cline{2-9}
 BG &  & $ t $ & 12.3776 & 36.2908 & 83.5639 & 139.8100 & 226.9250 & 325.4071 \\
\cline{2-9  }
  GC &  & $ t $ & 4.4151 &  6.3802 & 9.2605 & 12.5734 &  14.6543 & 16.655 \\
\hline
\end{tabular}}
\caption{$ t $ denotes the CPU time $ (s) $}
\end{table}

The following experiments are for corresponding exterior Dirichlet problem of Laplace equation: Table 7, 8.
\begin{table}
{\begin{tabular}{ccccccc}
\hline
& $ n $ &   $ 10^2 $ & $ 10^4 $ & $ 10^6 $ & $ 10^8 $& $ 10^{10} $  \\
\cline{2-7}
 $ d_1 = (1,0) $ & $ u_\infty $  &  6.3301 & 6.3310 & 6.3312 & 6.3314 & 6.3316 \\
\cline{2-7}
 & LS: $ \vert u_\infty - u^n_\infty \vert $  & $ 0.0324 e^{-3} $ &  $ 0.0680 e^{-3} $ & $ 0.1036 e^{-3} $ & $ 0.1391 e^{-3} $ & $ 0.1747 e^{-3} $ \\
\cline{2-7}
   & DLS: $ \vert u_\infty - u^n_\infty \vert $ & $ 0.0324 e^{-3} $ &  $ 0.0680 e^{-3} $ & $ 0.1036 e^{-3} $ & $ 0.1391 e^{-3} $ & $ 0.1747 e^{-3} $ \\
\cline{2-7}
 & BG: $ \vert u_\infty - u^n_\infty \vert $   & $ 0.0324 e^{-3} $ &  $ 0.0680 e^{-3} $ & $ 0.1036 e^{-3} $ & $ 0.1391 e^{-3} $ & $ 0.1747 e^{-3} $ \\
\cline{2-7}
 & GC: $ \vert u_\infty - u^n_\infty \vert $   & $ 0.0324 e^{-3} $ &  $ 0.0680 e^{-3} $ & $ 0.1036 e^{-3} $ & $ 0.1391 e^{-3} $ & $ 0.1747 e^{-3} $ \\
\hline
\end{tabular}}
\caption{$ u_\infty $ show the behavior of solution to Exterior problem of Laplace equation, the others are the error between solution and approximate solution at infinity.}
\end{table}

\begin{table}
{\begin{tabular}{ccccccc}
\hline
 & & $ n $ & $ 2 $ & $ 4 $ & $ 6 $ & $ 8 $   \\
\cline{2-7}
 & LS   & $ Err $ & $0.1208 $ & $ 0.0064 $ & $ 2.2890 e^{-5} $  & $ 2.2881 e^{-5} $      \\
 \cline{2-7}
& DLS   & $ Err $ & $0.1203 $ & $ 0.0064 $ & $ 2.2881 e^{-5} $  & $ 2.2881 e^{-5} $      \\
\cline{2-7}
 & BG & $ Err $ & $0.1203 $ & $ 0.0064 $ & $ 2.2881 e^{-5} $  & $ 2.2881 e^{-5} $      \\
 \cline{2-7}
 & GC & $ Err $ & $0.6812 $ & $ 0.0042 $ & $ 2.3354 e^{-4} $  & $ 2.2881 e^{-5} $      \\
\hline
\end{tabular}}
\caption{$ Err $ denotes the error between solution and approximate solutions near the boundary $ \partial \Omega $}
\end{table}

\begin{example}
 Let the boundary $ \partial \Omega $ of $ \Omega $ be parameterized by $ \gamma(t) = (e^{-1 + \cos t} , 2 e^{ -1 + \sin t}  ), \ t \in [0,2\pi]  $. We know that $ \Omega $ is bounded and simply connected with non-zero tangent vector in every point of $ \partial \Omega $. The $ \partial \Omega $ is described as
\begin{figure}[htbp]
\centering
 \includegraphics[width=0.8\textwidth]{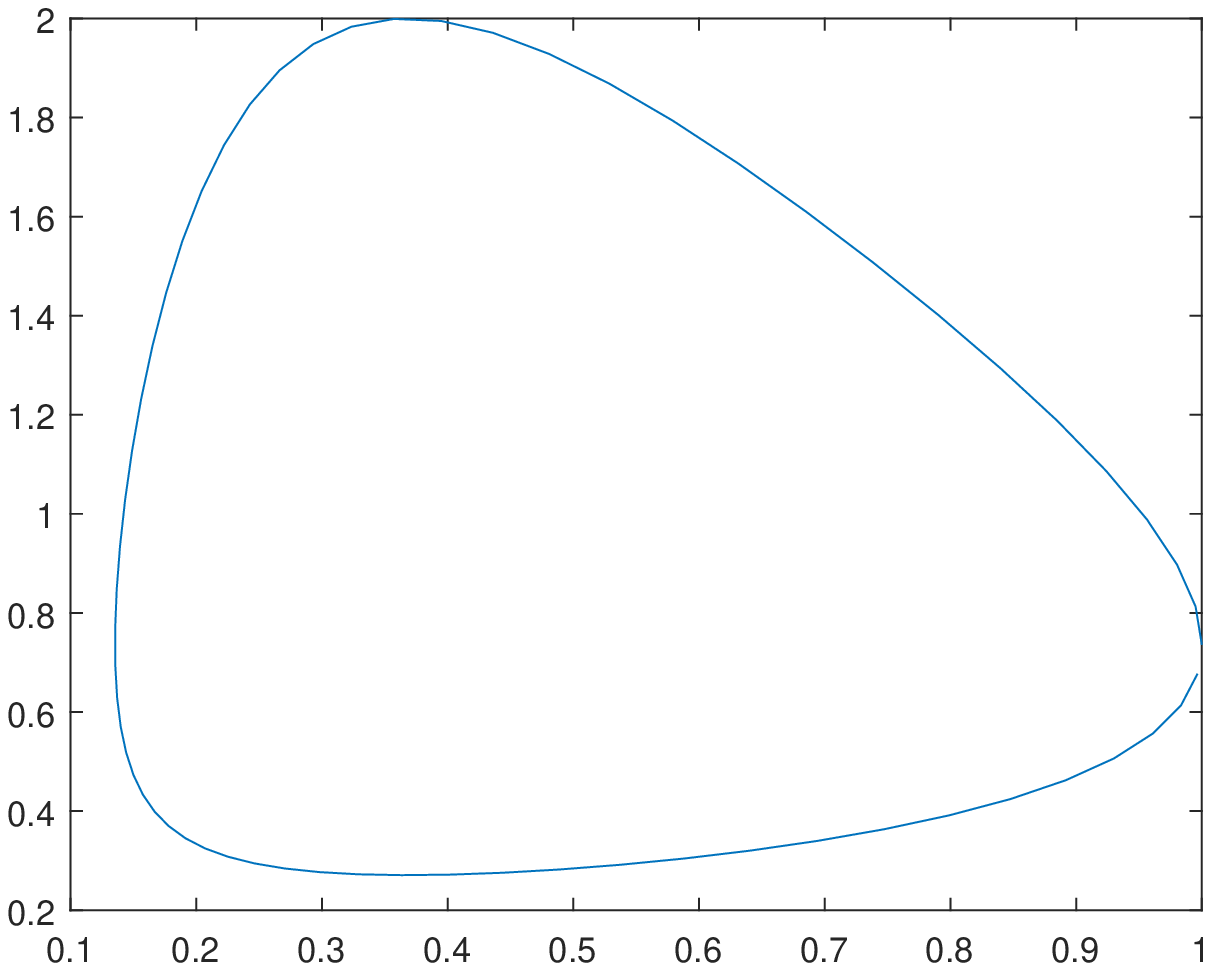}
\caption{$ \Omega $ for Example 5.2}
\end{figure}

Now
\begin{equation*}
k(t,s) =  -\frac{1}{2\pi} \ln \frac{  ( e^{-1 + \cos t} - e^{ -1 + \cos s} )^2 + (2 e^{ -1 + \sin t}   - 2 e^{-1 + \sin s} )^2   }{4 \sin^2 (\frac{t-s}{2})},  \quad t \neq s ,
 \end{equation*}
 \begin{equation*}
k (t,t) = -\frac{1}{2\pi} \ln ( e^{-2 +2 \cos t} \sin^2 t + 4 e^{-2 +  2\sin t } \cos^2 t  ), \quad 0 \leq t \leq 2\pi,
\end{equation*}

 \begin{equation*}
 2G_3(t) = - \frac{1}{ \vert \partial \Omega \vert} (-\frac{1}{2\pi} \int^{2\pi}_0 \ln (4 \sin^2 (\frac{t-\sigma}{2}))
 \end{equation*}
 \begin{equation*}
 \sqrt{  e^{-2 +2 \cos \sigma} \sin^2 \sigma + 4 e^{-2 +  2\sin \sigma } \cos^2 \sigma   } d\sigma
 \end{equation*}
 \begin{equation*}
  -\frac{1}{2\pi} \int^{2\pi}_0 \ln \frac{( \cos t - \cos \sigma )^2 + 4( \sin t  - \sin \sigma )^2   }{4 \sin^2 (\frac{t-\sigma}{2})}
  \end{equation*}
  \begin{equation*}
    \sqrt{  e^{-2 +2 \cos \sigma} \sin^2 \sigma + 4 e^{-2 +  2\sin \sigma } \cos^2 \sigma   }  d \sigma)
    \end{equation*}
    \begin{equation*}
  + \frac{2}{\vert \partial \Omega \vert}, \quad t \in [0,2\pi].
\end{equation*}
\end{example}

The following experiments are for modified Symm's integral equation of the first kind: Table 9-14.

\begin{table}
{\begin{tabular}{ccccccccc}
\hline
 & & $ n $ & $ 2 $ & $ 4 $ & $ 6 $ & $ 8 $ & 10 & 12  \\
\cline{2-9}
 & $ \delta = 0 $  & $ r $ & $ 0.0588 $ & $ 7.1196 e^{-4} $ & $ 4.7592 e^{-6} $  & $ 2.2762 e^{-6} $   & $ 2.2762 e^{-6} $  & $ 2.2762 e^{-6} $   \\
 \cline{2-9}
& $ \delta = 0.001 $  & $ r $ & $ 0.0588 $ & $ 7.1143 e^{-4} $ & $ 1.9955 e^{-4} $  & $ 0.0062 $   & $ 0.0062 $  & $ 0.0062 $  \\
\cline{2-9}
 & $ \delta = 0.01 $ & $ r $ & $ 0.0588 $ & $ 7.6825 e^{-4}  $ & $ 0.0020 $  & $ 0.00617 $   & $ 0.00617 $  & $ 0.00617 $ \\
 \cline{2-9}
 & $ \delta = 0.1 $  & $ r $ & $ 0.0589 $ & $ 0.0032 $ & $ 0.0200 $  & $ 0.6171 $   & $ 0.6173 $  & $ 0.6173 $ \\
\hline
\end{tabular}}
\caption{Least squares method for Example 5.2}
\end{table}
\begin{table}
{\begin{tabular}{ccccccccc}
\hline
 & & $ n $ & $ 2 $ & $ 4 $ & $ 6 $ & $ 8 $ & 10 & 12  \\
\cline{2-9}
 & $ \delta = 0 $  & $ r $ & $ 0.0617 $ & $ 0.0017 $ & $ 1.3082 e^{-5} $  & $ 3.0236 e^{-6} $   & $ 2.7641 e^{-6} $  & $ 1.0268 e^{-5} $   \\
\hline
\end{tabular}}
\caption{Dual least squares method for Example 5.2. }
\end{table}
\begin{table}
{\begin{tabular}{ccccccccc}
\hline
 & & $ n $ & $ 2 $ & $ 4 $ & $ 6 $ & $ 8 $ & 10 & 12  \\
\cline{2-9}
 & $ \delta = 0 $  & $ r $ & $ 0.0585 $ & $ 7.1098 e^{-4} $ & $ 4.7584 e^{-6} $  & $ 2.2762 e^{-6} $   & $ 2.2762 e^{-6} $  & $ 2.2762 e^{-5} $   \\
 \cline{2-9}
& $ \delta = 0.001 $  & $ r $ & $ 0.0585 $ & $ 7.1098 e^{-4} $ & $ 4.7584 e^{-6} $  & $ 0.0062 $   & $ 0.0062 $  & $ 0.0062 $  \\
\cline{2-9}
 & $ \delta = 0.01 $ & $ r $ & $ 0.0585 $ & $ 7.1098 e^{-4}  $ & $ 4.7584 e^{-6} $  & $ 0.00617 $   & $ 0.00617 $  & $ 0.00617 $ \\
 \cline{2-9}
 & $ \delta = 0.1 $  & $ r $ & $ 0.0585 $ & $ 7.1098 e^{-4}  $ & $ 4.7584 e^{-6} $    & $ 0.6172 $   & $ 0.6173 $  & $ 0.6173 $ \\
\hline
\end{tabular}}
\caption{Bubnov-Galerkin method for Example 5.2. }
\end{table}
\begin{table}
{\begin{tabular}{ccccccccc}
\hline
 & & $ n $ & $ 2 $ & $ 4 $ & $ 6 $ & $ 8 $ & 10 & 12  \\
\cline{2-9}
 & $ \delta = 0 $  & $ r $ & $ 0.0610 $ & $ 8.2528 e^{-4} $ & $ 5.5614 e^{-6} $  & $ 2.2762 e^{-6} $   & $ 2.2762 e^{-6} $  & $ 2.2762 e^{-5} $   \\
 \cline{2-9}
& $ \delta = 0.001 $  & $ r $ & $ 0.0610 $ & $ 0.0027 $ & $ 5.5614 e^{-6} $  & $ 0.0062 $   & $ 0.0062 $  & $ 0.0062 $  \\
\cline{2-9}
 & $ \delta = 0.01 $ & $ r $ & $ 0.0610 $ & $ 0.0261  $ & $ 5.5614 e^{-6} $  & $ 0.00617 $   & $ 0.00617 $  & $ 0.00617 $ \\
 \cline{2-9}
 & $ \delta = 0.1 $  & $ r $ & $ 0.0610 $ & $ 0.2611  $ & $ 5.5614 e^{-6} $    & $ 0.6171 $   & $ 0.6173 $  & $ 0.6173 $ \\
\hline
\end{tabular}}
\caption{Galerkin-Collocation method for Example 5.2. }
\end{table}
\begin{table}
{\begin{tabular}{ccccccc}
\hline
 & &  & LS & DLS & BG & GC  \\
\cline{2-7}
  $ \delta = 0 $  & $ n = 20 $ & $ r $ & $ 2.2882 e^{-6} $ & $51.0699 $ & $ 2.2762 e^{-6} $  & $ 2.2762 e^{-6} $     \\
\hline
\end{tabular}}
\caption{Additional experiments for PG and GC methods on Example 5.2 with $ n = 20 $}
\end{table}

\begin{table}
{\begin{tabular}{ccccccccc}
\hline
& & $ n $ & $ 2 $ & $ 4 $ & $ 6 $ & $ 8 $& 10 & 12 \\
\cline{2-9}
 LS &   & $ t $ & 17.6847 & 66.6814 & 133.0920 & 246.4948 & 368.4255 & 519.1206 \\
\cline{2-9}
 DLS &  & $ t $ & 12.8491 &  36.1843 & 80.9006 & 137.3488 & 209.8712 & 305.0765 \\
\cline{2-9}
 BG &  & $ t $ & 11.7570 & 39.6350 & 87.0154 & 168.0406 & 236.5605 & 387.6846 \\
\cline{2-9  }
  GC &  & $ t $ & 4.5987 &  8.0825 & 10.8370 & 12.1855 &  15.2162 & 18.5333 \\
\hline
\end{tabular}}
\caption{$ t $ denotes the CPU time $ (s) $}
\end{table}

The following experiments are for corresponding exterior Dirichlet problem of Laplace equation: Table 15,16.
\begin{table}
{\begin{tabular}{ccccccc}
\hline
& $ n $ &   $ 10^2 $ & $ 10^4 $ & $ 10^6 $ & $ 10^8 $& $ 10^{10} $  \\
\cline{2-7}
 $ d_1 = (1,0) $ & $ u_\infty $  &  2.7535 & 2.7544 & 2.7545 & 2.7545 & 2.7545 \\
\cline{2-7}
 & LS: $ \vert u_\infty - u^n_\infty \vert $  & $ 0.1259 e^{-4} $ &  $ 0.2411 e^{-4} $ & $ 0.3561 e^{-4} $ & $ 0.4711 e^{-4} $ & $ 0.5861
 e^{-4} $ \\
\cline{2-7}
   & DLS: $ \vert u_\infty - u^n_\infty \vert $ & $ 0.1259 e^{-4} $ &  $ 0.2411 e^{-4} $ & $ 0.3561 e^{-4} $ & $ 0.4711 e^{-4} $ & $ 0.5861 e^{-4} $ \\
\cline{2-7}
 & BG: $ \vert u_\infty - u^n_\infty \vert $   & $ 0.1259 e^{-4} $ &  $ 0.2411 e^{-4} $ & $ 0.3561 e^{-4} $ & $ 0.4711 e^{-4} $ & $ 0.5861 e^{-4} $ \\
\cline{2-7}
 & GC: $ \vert u_\infty - u^n_\infty \vert $   & $ 0.1373 e^{-4} $ &  $ 0.2525 e^{-4} $ & $ 0.3675 e^{-4} $ & $ 0.4825 e^{-4} $ & $ 0.5975 e^{-4} $ \\
\hline
\end{tabular}}
\caption{$ u_\infty $ show the behavior of solution to Exterior problem of Laplace equation, the others are the error between solution and approximate solution at infinity.}
\end{table}

\begin{table}
{\begin{tabular}{ccccccc}
\hline
 & & $ n $ & $ 2 $ & $ 4 $ & $ 6 $ & $ 8 $   \\
\cline{2-7}
 & LS   & $ Err $ & $0.0019 $ & $ 1.1127 e^{-5} $ & $ 9.4236 e^{-6} $  & $ 9.4233 e^{-6} $      \\
 \cline{2-7}
& DLS   & $ Err $ & $0.0039 $ & $ 3.3159e^{-5} $ & $ 9.4233 e^{-6} $  & $ 9.4233 e^{-6} $      \\
\cline{2-7}
 & BG & $ Err $ & $0.0013 $ & $ 1.1366 e^{-5} $ & $ 9.4233 e^{-6} $  & $ 9.4233 e^{-6} $      \\
 \cline{2-7}
 & GC & $ Err $ & $0.0093 $ & $ 1.0659 e^{-5} $ & $ 9.4233 e^{-6} $  & $ 9.4233 e^{-6} $      \\
\hline
\end{tabular}}
\caption{$ Err $ denotes the error between solution and approximate solutions near the boundary $ \partial \Omega $. }
\end{table}

With observations on above numerical results, we see that
\begin{itemize}
\item the dual least squares methods is instable and should not be chosen for practical computation, see Table 5, 13.
\item the least squares method, Bubnov-Galerkin, Galerkin-Collocation methods perform well with almost the same precision.
\item The Galerkin-Collocation methods is much more efficient than PG methods.
\item all approximate solutions induced by PG and GC methods to exterior Dirichlet problem satisfy the uniform convergence at infinity.
\end{itemize}
\section{Conclusion}
 In this paper, on the assumption that the boundary $ \partial \Omega $ is analytic, we extend the regular error analysis result from Symm's integral equation to modified Symm's integral equation. Besides, inheriting the quadrature technique to handle the singularity in Symm's integral equation, we establish the numerical procedures of Petrov-Galerkin and Galerkin-Collocation methods  with trigonometric interpolation basis on (1.7). We compare the efficiency, precision between above four numerical methods, show the instability of dual least squares method, and examine the precision of single layer approach with different projection methods to exterior Dirichlet problem of Laplace equation and the behavior of solution at infinity.

\section*{Acknowledgement}
The author thank Nianci Wu for helpful discussion on numerical implementation of modified Symm's integral equation.

\section*{References}
\bibliographystyle{elsarticle-num-names.bst}

\end{document}